\documentclass{amsart}
\usepackage{amsfonts}
\usepackage{amsmath}
\usepackage{amssymb}
\usepackage{graphicx}

\providecommand{\U}[1]{\protect\rule{.1in}{.1in}}
\newtheorem{theorem}{Theorem}
\theoremstyle{plain}
\newtheorem{acknowledgement}{Acknowledgement}

\newtheorem{corollary}{Corollary}

\newtheorem{lemma}{Lemma}

\newtheorem{proposition}{Proposition}

\numberwithin{equation}{section}
\input{tcilatex}

\begin{document}
\title[On van der Corput property of shifted primes]{On van der Corput
property of shifted primes}
\author{Sini\v{s}a Slijep\v{c}evi\'{c}}
\address{Department of Mathematics, Bijeni\v{c}ka 30, Zagreb,\ Croata}
\email{slijepce@math.hr}
\urladdr{}
\date{October 27, 2011}

\begin{abstract}
We prove that the upper bound for the van der Corput property of the set of
shifted primes is $O((\log n)^{-1+o(1)})$, giving an answer to a problem
considered by Ruzsa and Montgomery for the set of shifted primes $p-1$. We
construct normed non-negative valued cosine polynomials with the spectrum in
the set $p-1$, $p\leq n$, and a small free coefficient $a_{0}=O((\log
n)^{-1+o(1)})$. This implies the same bound for the intersective property of
the set $p-1$, and also bounds for several properties related to uniform
distribution of related sets.
\end{abstract}

\maketitle

\section{Introduction}

We say that a set of integers $\mathcal{S}$ is a \textit{van der Corput} (or 
\textit{correlative}) set, if given a real sequence $(x_{n})_{n\in N}$, if
all the sequences $(x_{n+d}-x_{n})_{n\in N}$, $d\in \mathcal{S}$, are
uniformly distributed $\func{mod}1$, then the sequence $(x_{n})_{n\in N}$ is
itself uniformly distributed $\func{mod}1$. The property was introduced by
Kamae and Mend\`{e}s France (\cite{Kamae:77}), and is important as it is
closely related to the intersective property of integers, discussed below.
Classical examples of van der Corput sets are sets of squares, shifted
primes $p+1$, $p-1$, and also sets of values $P(n)$, where $P$ is any
polynomial with integer coefficients, and has a solution of $P(n)\equiv 0(%
\func{mod}k$) for all $k$. All van der Corput sets are intersective sets,
but the converse does not hold, as was shown by Bourgain (\cite{Bourgain:87}%
).

We first recall the key characterization of the van der Corput property. If $%
\mathcal{S}$ is a set of positive integers, then let $\mathcal{S}_{n}=%
\mathcal{S}\cap \{1,...,n\}$. We denote by $\mathcal{T}(\mathcal{S})$ the
set of all cosine polynomials%
\begin{equation}
T(x)=a_{0}+\tsum_{d\in \mathcal{S}_{n}}a_{d}\cos (2\pi dx)\text{,}
\label{r:d0}
\end{equation}%
$T(0)=1$, $T(x)\geq 0$ for all $x$, where $n$ is any integer and $a_{0}$, $%
a_{d}$ are real numbers (i.e. $T$ is a non-negative normed cosine polynomial
with the spectrum in $\mathcal{S}\cup \{0\}$). Kamae and Mend\`{e}s France
proved that a set is a van der Corput set if and only if (\cite{Kamae:77}, 
\cite{Montgomery:94})%
\begin{equation}
\inf_{T\in \mathcal{T}(\mathcal{S})}a_{0}=0.  \label{r:d2}
\end{equation}

We can define a function which measures how quickly a set is becoming a van
der Corput set with%
\begin{equation}
\gamma (n)=\inf_{T\in \mathcal{T}(\mathcal{S}_{n})}a_{0},  \label{r:dgamma}
\end{equation}%
and then a set is van der Corput if and only if $\gamma (n)\rightarrow 0$ as 
$n\rightarrow \infty $.

Ruzsa and Montgomery set a problem of finding any upper bound for the
function $\gamma $ for any non-trivial van der Corput set (\cite%
{Montgomery:94}, unsolved problem 3; \cite{Ruzsa:84a}). Ruzsa in \cite%
{Ruzsa:81} announced the result that for the set of squares, $\gamma
(n)=O((\log n)^{-1/2})$, but the proof was never published. The author in 
\cite{Slijepcevic:09} proved that for the set of squares, $\gamma
(n)=O((\log n)^{-1/3})$. In this paper we prove the following result:

\begin{theorem}
\label{t:main01}If $\mathcal{S}$ is the set of shifted primes $p-1$, then $%
\gamma (n)=O((\log n)^{-1+o(1)})$.
\end{theorem}

The gap between the upper bound and the best available lower bound remains
very large, as in the case of the sets of recurrence discussed below. The
lower bound below relies on a construction of Ruzsa \cite{Ruzsa84b}:

\begin{theorem}
\label{t:main02}If $\mathcal{S}$ is the set of shifted primes $p-1$, then $%
\gamma (n)\,\gg n^{\left( -1+\frac{\log 2-\varepsilon }{\log \log n}\right)
} $, where $\varepsilon >0$ is an arbitrary real number.
\end{theorem}

\textbf{Structure of the proof and its limitations.} We define a cosine
polynomial%
\begin{equation}
F_{N,d}(\theta )=\frac{1}{k}\func{Re}\sum_{\substack{ p\leq dN+1  \\ p\equiv
1(\func{mod}d)}}\log p\cdot e((p-1)\theta ),  \label{r:defF}
\end{equation}%
where $e(\theta )=\exp (2\pi i\theta )$ and $k$ is chosen so that $%
F_{N,d}(0)=1$. We show in Sections 2 and 3 by using exponential sum
estimates along major and minor arcs that%
\begin{equation*}
F_{N,d}(\theta )\geq \tau (d,q)+E(d,q,\kappa ,N).
\end{equation*}%
Here $\kappa =\theta -a/q$, the function $E$ is the error term and $\tau
(d,q)$ is the principal part which is (for square-free $d$) $1$ for $q|d$, $%
0 $ if $q$ not square-free, and $-1/\varphi (q/(q,d))$ otherwise ($\varphi $
being the Euler's totient function and $(q,d)$ the greatest common divisor).
In Section 4 we demonstrate that for a given $\delta >0$, one can find a
collection of positive integers $\mathcal{D}$ not exceeding $\exp ((\log
1/\delta )^{2+o(1)})$ and weights $\sum_{d\in \mathcal{D}}w(d)=1$ such that
for any integer $q>0$,%
\begin{equation*}
\sum_{d\in \mathcal{D}}w(d)\tau (d,q)\geq -\delta /2\text{.}
\end{equation*}%
In addition, one can find constants $R,N$ not exceeding $O(\exp ((\log
1/\delta )^{4+o(1)}))$ for any given $\theta $ such that if $a/q$ is the
Dirichlet's approximation of $\theta =a/q+\kappa $, $\kappa \leq 1/(qR)$,
then the error term $|E(d,q,\kappa ,N)|\leq \delta /2$. This seemingly
implies effectively the same upper bound for $\gamma (n)$ as obtained in 
\cite{Ruzsa:08} for a stronger intersective property of sets of integers
(see below).

Unfortunately, in our calculations the constants $R,N$ can not be chosen so
that for all $\theta \in \boldsymbol{T}=\boldsymbol{R}/\boldsymbol{Z}$ the
error term is small. Namely, for $d\theta $ close to an integer, the error
term is $O(dN/R)$, and for $\theta $ on minor arcs, the error term is $%
O(d^{2}\sqrt{R}/\sqrt{N})$. We resolve it by choosing a geometric sequence
of constants $N_{1},...,N_{4/\delta }$, which results with the bound in
Theorem \ref{t:main01}. We finalize the proof in Section 5 by constructing
the required cosine polynomial as a convex combination of $F_{N,d}$ over $%
d\in \mathcal{D}$ and $N_{j}$.

\textbf{Applications.}We say a set $\mathcal{S}$ is \textit{intersective set}
(or a \textit{set of recurrence}, or a \textit{Poincar\'{e} set)}, if for
any set $A$ of integers with positive upper Banach density 
\begin{equation*}
\rho (A)=\lim \sup_{n\rightarrow \infty }|A\cap \lbrack 1,n]|/n>0,
\end{equation*}%
its difference set $A-A$ contains an element of $\mathcal{S}$. Given any set
of integers $\mathcal{S}$, one can define the function $\alpha :\boldsymbol{N%
}\rightarrow \lbrack 0,1]$ as $\alpha (n)=\sup \rho (A)$, where $A$ goes
over all sets of integers whose difference set does not contain an element
of $\mathcal{S}\cap \lbrack 1,n]$ (equivalent definitions of $\alpha $ can
be found in \cite{Ruzsa:84a}). A set is an intersective set if and only if $%
\lim_{n\rightarrow \infty }\alpha (n)=0$. Ruzsa in \cite{Ruzsa:84a} also
proved that if $\mathcal{S}$ is a van der Corput set, then it is also an
intersective set, and%
\begin{equation*}
\alpha (n)\leq \gamma (n).
\end{equation*}

The bound $\alpha (n)=O((\log n)^{-1+o(1)})=O(\exp ((-1+o(1))\log \log n))$
for the set of shifted primes follows then as a corollary of Theorem \ref%
{t:main01}. This is worse than the bound $\alpha (n)=O(\exp (-c\sqrt[4]{\log
n}))$ obtained by Ruzsa and Sanders in \cite{Ruzsa:08}, but better than
earlier bounds in \cite{Lucier:08} and \cite{Sarkozy:78b}.

The function $\gamma (n)$ has different characterizations and further
applications discussed in detail in \cite{Montgomery:94}. We discuss in
Section 9 the Heilbronn property of the set of shifted primes, which
specifies how well the expression $x(p-1)$ can approximate integers
uniformly in $x\in \boldsymbol{R}$, by choosing for a given $x$ some prime $%
p\leq n$ so that $x(p-1)$ is as close to an integer as possible.

\section{The major arcs}

If $\Lambda $ is the von-Mangoldt function, we define as in \cite{Ruzsa:08}%
\begin{equation*}
\Lambda _{N,d}(x):=\left\{ 
\begin{array}{cc}
\Lambda (dx+1) & \text{if }1\leq x\leq N \\ 
0 & \text{otherwise,}%
\end{array}%
\right.
\end{equation*}%
and let $\Lambda _{N}(x)=\Lambda _{N,1}(x)$. The Fourier transform $\widehat{%
.}:l^{1}(\boldsymbol{Z})\rightarrow L^{\infty }(\boldsymbol{R})$ is defined
as the map which takes $f\in l^{1}(\boldsymbol{Z})$ to $\widehat{f}(\theta
)=\sum_{x\in \boldsymbol{Z}}f(x)\overline{e(x\theta )}$, thus $\widehat{%
\Lambda _{N,d}}(\theta )$ is the exponential sum%
\begin{equation*}
\widehat{\Lambda _{N,d}}(\theta )=\sum_{x\leq N}\Lambda (dx+1)\overline{%
e(x\theta )}\text{.}
\end{equation*}

The classical estimates for Fourier transforms of $\Lambda _{N,d}(x)$ were
optimized by Ruzsa and Sanders to the class of problems studied in this
paper. They studied two cases related to the generalized Riemann hypothesis:
given a pair of integers $D_{1}\geq D_{0}\geq 2$, then there either exists
an exceptional Dirichlet character of modulus $d_{D}$ $\leq D_{0}$ or not (%
\cite{Ruzsa:08}, Proposition 4.7). They then obtained the following
estimates (we will be more specific below on the assumptions):\ if $\kappa
=\theta -a/q$, where $\theta \in \boldsymbol{T}$, then%
\begin{eqnarray}
\left\vert \widehat{\Lambda _{N,d}}(\theta )\right\vert &\leq &\frac{|\tau
_{a,d,q}|}{\varphi (q)}\widehat{\Lambda _{N,d}}(0)+O\left( (1+|\kappa
|N)E_{N,D_{1}}\right) \text{,}  \label{r:rs1} \\
\left\vert \widehat{\Lambda _{N,d}}(0)\right\vert &\gg &\frac{N}{\varphi (d)}%
+O\left( E_{N,D_{1}}\right) \text{,}  \label{r:rs2}
\end{eqnarray}%
where%
\begin{eqnarray*}
E_{N,D_{1}} &=&ND_{1}^{2}\exp \left( -\frac{c_{1}\log N}{\sqrt{\log N}+\log
D_{1}}\right) , \\
\tau _{a,d,q} &=&\sum_{\substack{ m=0  \\ (md+1,q)=1}}^{q-1}e\left( m\frac{a%
}{q}\right) \text{.}
\end{eqnarray*}

\begin{proposition}
\label{p:rusa}(Ruzsa, Sanders). There is an absolute constant $c_{1}$ such
that for any pair of integers $D_{1}\geq D_{0}\geq 2$, one of the following
possibilities hold:

(i)\ ($(D_{1},D_{0})$ is exceptional). There is an integer $d_{D}\leq D_{0}$%
, such that for all non-negative integers $N,a,q,d$, where $1\leq dq\leq
D_{1}$, $d_{D}|d$, and $(a,q)=1$, for any $\theta \in \boldsymbol{T}$ (\ref%
{r:rs1}), (\ref{r:rs2}) hold, where $\kappa =\theta -a/q$.

(ii)\ ($(D_{1},D_{0})$ is unexceptional). For all non-negative integers $%
N,a,q,d$, where $1\leq dq\leq D_{0}$ and $(a,q)=1$, for any $\theta \in 
\boldsymbol{T}$ (\ref{r:rs1}), (\ref{r:rs2}) hold, where $\kappa =\theta
-a/q $.
\end{proposition}

\begin{proof}
(\cite{Ruzsa:08}), Propositions 5.3. and 5.5. (Note that (\ref{r:rs1}) is
explicitly obtained at the end of the proof of Proposition 5.3.)
\end{proof}

We now define a function $\tau $ closely related to $\tau _{a,d,q}$ above,
which will be the main term when estimating cosine polynomials $F_{N,d}$. Let%
\begin{equation}
\tau (d,q)=\left\{ 
\begin{array}{ll}
1, & q|d \\ 
0, & (d,r)>1\text{ or }r\text{ not square-free} \\ 
-1/\varphi (r) & \text{otherwise,}%
\end{array}%
\right.  \label{r:deftau}
\end{equation}%
where $r=q/(q,d)$. Note that for $d$ square-free, the second row condition
above is equivalent to $q$ being not square-free.

\begin{lemma}
\label{l:tau}Let $a,d,q$ be positive integers, $(a,q)=1$, $r=q/(q,d)>1$ and $%
a^{\ast }=ad/(q,d)$. Then%
\begin{equation}
\frac{|\tau _{a^{\ast },d,r}|}{\varphi (r)}=|\tau (d,q)|\text{.}
\label{r:help}
\end{equation}
\end{lemma}

\begin{proof}
As was noted in \cite{Ruzsa:08}, Section 5,%
\begin{equation*}
\tau _{a,d,q}=\left\{ 
\begin{array}{ll}
c_{q}(a)e(-m_{d,q}a/q) & \text{if }(d,q)=1 \\ 
0 & \text{otherwise,}%
\end{array}%
\right.
\end{equation*}%
where $c_{q}(a)$ is the Ramanujan sum and $m_{d,q}$ is a solution of $%
m_{d,q}d\equiv 1(\func{mod}q)$. Now if $q|d$, $\tau _{a^{\ast },d,r}=\tau
_{a^{\ast },d,1}=1$, thus both sides of (\ref{r:help}) are equal to $1$. If $%
(d,r)>1$, then $\tau _{a^{\ast },d,r}=0$, and if $r$ not square-free, then $%
\tau _{a^{\ast },d,r}=0$ as the Ramanujan sum $c_{r}(a^{\ast })=0$ when $r$
not square-free. The remaining case follows from $(a^{\ast },r)=1$, $r$
square-free implying that the Ramanujan sum $|c_{r}(a^{\ast })|=1$.
\end{proof}

It is easy to see that there exists a constant $c_{2}$ depending only on $%
c_{1}$ such that if%
\begin{equation}
\log N\geq c_{2}(\log D_{1})^{2},  \label{r:new1}
\end{equation}%
then%
\begin{equation}
D_{1}^{2}\exp \left( -\frac{c_{1}\log N}{\sqrt{\log N}+\log D_{1}}\right)
\leq \frac{1}{D_{1}^{2}}\text{.}  \label{r:new2}
\end{equation}

We first discuss the case of $q$ not dividing $d$, and then $q|d$.

\begin{proposition}
\label{p:major}Assume all the assumptions of Proposition \ref{p:rusa} hold
for $D_{0},D_{1},\theta ,N,a,q,d,\kappa $, and in addition (\ref{r:new1}), (%
\ref{r:new2}). If $q$ not dividing $d$, then%
\begin{equation*}
F_{N,d}(\theta )\geq \tau (d,q)+O\left( \frac{1}{D_{1}}+|\kappa |N\right) 
\text{.}
\end{equation*}
\end{proposition}

\begin{proof}
If we write%
\begin{eqnarray*}
\psi (x;q,a) &=&\sum_{\substack{ n\leq x  \\ n\equiv a(\func{mod}q)}}\Lambda
(n), \\
\vartheta (x;q,a) &=&\sum_{\substack{ p\leq x  \\ p\equiv a(\func{mod}q)}}%
\log (p),
\end{eqnarray*}%
then $\widehat{\Lambda _{N,d}}(0)=\psi (Nd+1;d,1)$ and $k=\vartheta
(Nd+1;d,1)$ where $k$ is the denominator in (\ref{r:defF}). By the
well-known property of functions $\psi ,\vartheta $ (see e.g. \cite%
{Montgomery:07}, p.381),%
\begin{equation*}
\psi (Nd+1;d,1)-\vartheta (Nd+1;d,1)\ll \sqrt{dN}\text{.}
\end{equation*}%
Relations (\ref{r:rs2}), (\ref{r:new2}) and $\varphi (d)<D_{1}$ imply that%
\begin{equation}
\frac{N}{\left\vert \widehat{\Lambda _{N,d}}(0)\right\vert }\ll D_{1}\text{.}
\label{r:lambda0}
\end{equation}%
If we use the shorthand notation $F=\func{Re}\sum_{_{p\leq dN+1,p\equiv 1(%
\func{mod}d)}}\log p\cdot e((p-1)\theta )$, and then $F_{N,d}(\theta )=F/k$,
we see from definitions that $F$ is approximately $\func{Re}\widehat{\Lambda
_{N,d}}(d\theta )$, or more precisely 
\begin{equation*}
|\func{Re}\widehat{\Lambda _{N,d}}(d\theta )-F|\leq \widehat{\Lambda _{N,d}}%
(0)-k\ll \sqrt{dN}.
\end{equation*}%
Putting these three inequalities together,%
\begin{equation}
\left\vert \frac{F}{k}-\frac{\func{Re}\widehat{\Lambda _{N,d}}(d\theta )}{%
\widehat{\Lambda _{N,d}}(0)}\right\vert \leq \left\vert \frac{F}{k}%
\right\vert \frac{|\widehat{\Lambda _{N,d}}(0)-k|}{|\widehat{\Lambda _{N,d}}%
(0)|}+\frac{|\func{Re}\widehat{\Lambda _{N,d}}(d\theta )-F|}{|\widehat{%
\Lambda _{N,d}}(0)|}\ll \frac{\sqrt{d}}{\sqrt{N}}D_{1}\text{.}
\label{r:relf}
\end{equation}%
Now if $\theta -a/q=\kappa $, then $d\theta -a^{\ast }/r=d\kappa $, where $%
a^{\ast }=ad/(d,q)$, $r=q/(d,q)$. Combining (\ref{r:rs1}), (\ref{r:rs2}), (%
\ref{r:new2}) and (\ref{r:lambda0}) we easily get that%
\begin{equation*}
\left\vert \frac{\widehat{\Lambda _{N,d}}(d\theta )}{\widehat{\Lambda _{N,d}}%
(0)}\right\vert \leq \frac{|\tau _{a^{\ast },d,r}|}{\varphi (r)}+O\left( 
\frac{1}{D_{1}}+|\kappa |N\right) \text{.}
\end{equation*}%
The last two relations combined (noting that if $d\leq D_{1}$ and (\ref%
{r:new1}), then $\sqrt{d}D_{1}/\sqrt{N}\ll 1/D_{1}$)\ and Lemma \ref{l:tau}
complete the proof.
\end{proof}

\begin{proposition}
\label{p:main2}Say $d,N$ are\ positive integers, $\theta \in \boldsymbol{T}$%
, and $\kappa =\theta -a/q$, $(a,q)=1$ and $q|d$. Then%
\begin{equation}
F_{N,d}(\theta )\geq 1+O(dN|\kappa |)\text{.}  \label{p:m2}
\end{equation}
\end{proposition}

\begin{proof}
We first recall that $\func{Re}e(\theta )=\cos (2\pi \theta )\geq 1-2\pi
\left\Vert \theta \right\Vert $, where $\left\Vert .\right\Vert $ is the
distance from the nearest integer. Thus if $|dN\kappa |\leq 1/2,$ then for
each $p\leq dN+1$, $d|(p-1)$, we get $\left\Vert (p-1)\theta \right\Vert
=(p-1)|\kappa |$ and $\func{Re}e((p-1)\theta )\geq 1-2\pi dN|\kappa |$,
which easily implies (\ref{p:m2}).
\end{proof}

\section{The minor arcs}

We start with the minor arc estimate from \cite{Ruzsa:08}, Corollary 6.2,
which is derived from the classical result of Vinogradov (\cite%
{Montgomery:94}, Theorem 2.9).

\begin{proposition}
\label{t:vin}Suppose that $d\leq N$ and $q\leq R$ are positive integers, $%
\theta \in \boldsymbol{T}$, $(a,q)=1$ and $|\theta -a/q|\leq 1/qR$. Then%
\begin{equation}
\left\vert \widehat{\Lambda _{N,d}}(\theta )\right\vert \ll d(\log
N)^{4}\left( \frac{N}{\sqrt{q}}+N^{4/5}+\sqrt{NR}\right) \text{.}
\label{r:a1}
\end{equation}
\end{proposition}

The minor arc estimate for $F_{N,d}(\theta )$ now follows.

\begin{corollary}
\label{c:main3}Suppose $d\leq D_{1}$, $q\leq R,$ $N$ are positive integers, $%
\theta \in \boldsymbol{T}$, $(a,q)=1$ and $|\theta -a/q|\leq 1/qR$. Assume
also (\ref{r:new1}) and (\ref{r:new2}) hold. Then%
\begin{equation}
|F_{d,N}(\theta )|\ll D_{1}^{2}(\log N)^{4}\left( \frac{1}{\sqrt{q}}%
+N^{-1/5}+\frac{\sqrt{R}}{\sqrt{N}}\right) \text{.}  \label{p:m3}
\end{equation}
\end{corollary}

\begin{proof}
First note that as $d\leq D_{1}$, Proposition \ref{p:rusa} implies that (\ref%
{r:rs2}) holds. Then similarly as in the proof of Proposition \ref{p:major},%
\begin{equation}
\frac{N}{\left\vert \widehat{\Lambda _{N,d}}(0)\right\vert }\ll D_{1}
\label{r:a2}
\end{equation}%
and%
\begin{equation}
\left\vert \frac{F}{k}-\frac{\func{Re}\widehat{\Lambda _{N,d}}(d\theta )}{%
\widehat{\Lambda _{N,d}}(0)}\right\vert \ll \frac{\sqrt{d}}{\sqrt{N}}%
D_{1}\leq \frac{D_{1}^{3/2}}{\sqrt{N}}\text{.}  \label{r:a3}
\end{equation}%
We complete the proof by combining (\ref{r:a1}), (\ref{r:a2}) and (\ref{r:a3}%
).
\end{proof}

\section{Cancelling out the main term}

Recall the definition of the arithmetic function $\tau $ in (\ref{r:deftau}%
). We first cancel out the main terms in the unexceptional case.

\begin{theorem}
\label{t:cancelling}For a given $\delta >0$ smaller than some $\delta _{0}>0$
there exists a collection of positive integers $\mathcal{D}$ not greater
than $\exp ((\log 1/\delta )^{2+o(1)})$ and weights $w:\mathcal{D\rightarrow 
}\boldsymbol{R}$, $\sum_{d\in \mathcal{D}}w(d)=1$, such that for all
positive integers $q$,%
\begin{equation}
\sum_{d\in \mathcal{D}}w(d)\tau (d,q)\geq -\delta /2.  \label{r:A}
\end{equation}
\end{theorem}

\begin{proof}
We first define the set $\mathcal{D}$ depending on three constants $%
p^{-}<p^{+}$, $l$ to be defined below. Let 
\begin{equation*}
d^{\ast }=\prod_{p\leq p^{-}}p
\end{equation*}%
($p$ denoting a product over primes as usual), and let $\mathcal{D}(j)$ be
the set of all square-free numbers $d^{\ast }d$, $d$ containing in its
decomposition only primes $p^{-}<p\leq p^{+}$, and such that $\omega (d)=j$,
where $\omega (d)$ denotes the number of distinct primes dividing $d$. We
set now%
\begin{eqnarray*}
p^{+} &=&2/\delta +1, \\
l &=&\left\lceil 2\log (1/\delta )\left( \frac{2\log \log (2/\delta )}{\log 2%
}+1\right) \right\rceil =\log (1/\delta )^{1+o(1)}, \\
p^{-} &=&2l^{2}+1=\log (1/\delta )^{2+o(1)}, \\
\mathcal{D} &=&\mathcal{D}(l), \\
W(j) &=&\sum_{d^{\ast }d\in \mathcal{D}(j)}1/\varphi (d), \\
w(d^{\ast }d) &=&\frac{1}{W(l)}\frac{1}{\varphi (d)}\text{,}
\end{eqnarray*}%
where $\left\lceil x\right\rceil $ is the smallest integer $\geq x$. We
denote the left-hand side of (\ref{r:A}) with $A(q)$.

By using $\prod_{p\leq x}p=\exp (x^{(1+o(1))})$ (see e.g. \cite%
{Montgomery:07}, Corollary 2.6), we easily see that for each $d^{\ast }d\in 
\mathcal{D}$,%
\begin{equation*}
d^{\ast }d\leq \prod_{p\leq p^{-}}p\cdot (p^{+})^{l}=\exp (\log (1/\delta
)^{2+o(1)}).
\end{equation*}

If $q$ is not square-free or $q$ contains a prime larger than $p^{+}$, the
claim $A(q)\geq -\delta /2$ is straightforward as for all $d$, $\tau (d,q)=0$%
, respectively $\tau (d,q)\geq -1/\varphi (p^{+})\geq -\delta /2$.

We can now without loss of generality assume that $q$ is square-free,
containing no prime $>p^{+}$ or $\leq p^{-}\,$in its decomposition (the
latter can be eliminated as primes $\leq p^{-}\,\ $do not affect the value
of $\tau (d^{\ast }d,q)$ for square-free $q$). We define the following
constants and sets to assist us in calculations:%
\begin{eqnarray*}
k &=&\log (1/\delta ), \\
\mathcal{D}(j;q) &=&\{d^{\ast }d\in \mathcal{D}(j)\text{, }(d,q)=1\}, \\
W(j;q) &=&\sum_{d^{\ast }d\in \mathcal{D}(j,q)}1/\varphi (d), \\
W &=&W(1)=\sum_{p^{-}<p\leq p^{+}}\frac{1}{\varphi (p)}=\sum_{p^{-}<p\leq
p^{+}}\frac{1}{p-1}\text{.}
\end{eqnarray*}

The remaining cases will be distinguished by $\omega (q)$.

(i) Assume $\omega (q)\leq 2k$. We will show that the terms for which $q|d$
dominate all the others. We first show the following: for $j_{1}<j_{2}$, 
\begin{equation}
W(j_{2};q)\leq \frac{W^{j_{2}-j_{1}}W(j_{1};q)}{j_{2}(j_{2}-1)...(j_{1}+1)}%
\text{.}  \label{r:laminq}
\end{equation}%
Indeed, if we define%
\begin{equation*}
W^{\ast }(j;q)=\sum_{(p_{1},p_{2},...,p_{j})}\frac{1}{\varphi
(p_{1}p_{2}...p_{j})}\text{,}
\end{equation*}%
where the sum goes over all ordered j-tuples of pairwise different primes $%
p_{i}$, $p^{-}<p_{i}\leq p^{+},\,\ p_{i}$ coprime with $q$, then $%
W(j;q)=W^{\ast }(j;q)/j!$. However, as $\varphi $ is multiplicative for
coprime integers,%
\begin{equation}
W^{\ast }(j_{2};q)\leq W^{j_{2}-j_{1}}W^{\ast }(j_{1};q)  \label{r:laminq2}
\end{equation}%
(we first choose the first $j_{2}-j_{1}$ primes and then the remaining $j_{1}
$). We obtain (\ref{r:laminq}) by dividing (\ref{r:laminq2}) with $j_{2}!$.

The definition of $A(q)$ now yields:%
\begin{equation*}
A(q)=\sum_{q|d}\frac{1}{\varphi (d)}-\sum_{q/(q,d)>1}\frac{1}{\varphi (d)}%
\frac{1}{\varphi (r)},
\end{equation*}%
where the sums above and below are over $d^{\ast }d\in \mathcal{D}$ unless
specified otherwise and $r$ always denotes $r=q/(q,d)$ (recall that we
assumed that $q$ and $d^{\ast }$ are coprime). We first detail out the first
term:%
\begin{equation*}
\sum_{q|d}\frac{1}{\varphi (d)}=\sum_{d^{\ast }d\in \mathcal{D}(l-\omega
(q);q)}\frac{1}{\varphi (d)}\frac{1}{\varphi (q)}=W(l-\omega (q);q)\frac{1}{%
\varphi (q)}.
\end{equation*}

If $\omega ((d,q))=j$, we can choose $(d,q)$ as a factor of $q$ in $\binom{%
\omega (q)}{j}$ ways. Using that, (\ref{r:laminq})\ and in the last rows $%
\omega (q)\leq 2k$ and $(1+x/n)^{n}<\exp (x)$ we obtain 
\begin{eqnarray*}
\sum_{q/(q,d)>1}\frac{1}{\varphi (d)}\frac{1}{\varphi (r)}
&=&\sum_{j=0}^{\omega (q)-1}\sum_{\omega ((d,q))=j}\frac{1}{\varphi (d)}%
\frac{1}{\varphi (r)}=\sum_{j=0}^{\omega (q)-1}W(l-j;q)\binom{\omega (q)}{j}%
\frac{1}{\varphi (q)}\leq  \\
&\leq &\sum_{j=0}^{\omega (q)-1}\frac{W^{\omega (q)-j}}{(l-j)...(l-\omega
(q)+1)}\binom{\omega (q)}{j}\cdot \frac{W(l-\omega (q);q)}{\varphi (q)}\leq 
\\
&\leq &\frac{W(l-\omega (q);q)}{\varphi (q)}\sum_{j=0}^{\omega (q)-1}\binom{%
\omega (q)}{j}\frac{W^{\omega (q)-j}}{(l-\omega (q))^{\omega (q)-j}}\leq  \\
&\leq &\frac{W(l-\omega (q);q)}{\varphi (q)}\left[ \left( 1+\frac{W}{(l-2k)}%
\right) ^{2k}-1\right] < \\
&<&\frac{W(l-\omega (q);q)}{\varphi (q)}\left[ \exp \left( \frac{W}{l/(2k)-1}%
\right) -1\right] .
\end{eqnarray*}%
As by e.g. \cite{Montgomery:07}, Theorem 2.7.(d),%
\begin{equation}
\sum_{p\leq x}\frac{1}{p-1}=\log \log x\cdot (1+o(1)),  \label{r:sump}
\end{equation}%
we get that 
\begin{equation}
W=\sum_{p^{-}<p\leq p^{+}}\frac{1}{p-1}=\log \log (p^{+})(1+o(1))\leq 2\log
\log (2/\delta ).  \label{r:below}
\end{equation}%
It is easy to check that the definitions of $l,k$ imply that%
\begin{equation*}
1-\left[ \exp \left( \frac{2\log \log (2/\delta )}{l/(2k)-1}\right) -1\right]
\geq 0\text{.}
\end{equation*}

Putting all of the above together we get $A(q)>0$.

(ii)\ Assume $2k<\omega (q)\leq 2l$. We now show that all the terms are
small. First assume $\omega ((q,d))=j\geq k$. By the same reasoning as in (%
\ref{r:laminq}) one gets for $j\leq l$, 
\begin{equation*}
W(l;q)=\frac{(W-\sum_{p|q}1/\varphi (p))^{l-j}W(j;q)}{l(l-1)...(j+1)}\text{.}
\end{equation*}

Now by definition, $W(l)\geq W(l;q)$. Applying again (\ref{r:sump}) we see
that for $\delta $ small enough,%
\begin{equation*}
W-\sum_{p|q}1/\varphi (p)\geq \log \log (p^{+})(1+o(1))-\log \log
(2l)(1+o(1))\geq 1\text{.}
\end{equation*}%
Combining all of it one gets%
\begin{equation*}
\frac{W(j;q)}{W(l)}\leq l^{l-j}\text{.}
\end{equation*}

Furthermore, as by the Stirling's formula $k!\geq k^{k}\exp (-k)$ and as $%
k=\log (1/\delta )$, we get for $\delta $ small enough 
\begin{equation*}
\frac{l}{k!}\leq \frac{\log (1/\delta )^{(1+o(1))}}{\log (1/\delta )^{\log
(1/\delta )}\exp (-(\log (1/\delta ))}\leq \delta /4.
\end{equation*}

Putting it all that together and summing over $d^{\ast }d\in \mathcal{D}$
similarly as above we get%
\begin{eqnarray}
\sum_{j=k}^{l}\sum_{\omega ((d,q))=j}|w(d^{\ast }d)\tau (d^{\ast }d,q)| &=&%
\frac{1}{W(l)}\sum_{j=k}^{\min \{l,\omega (q)\}}\sum_{\omega (d,q)=j}\frac{1%
}{\varphi (d)}\frac{1}{\varphi (r)}=  \notag \\
&=&\sum_{j=k}^{\min \{l,\omega (q)\}}\binom{\omega (q)}{j}\frac{W(l-j;q)}{%
W(l)}\frac{1}{\varphi (q)}\leq  \notag \\
&\leq &\sum_{j=k}^{\min \{l,\omega (q)\}}\frac{(2l)^{j}}{j!}l^{j}\frac{1}{%
(p^{-}-1)^{\omega (q)}}\leq  \notag \\
&\leq &\frac{1}{k!}\sum_{j=k}^{\min \{l,\omega (q)\}}\left( \frac{2l^{2}}{%
p^{-}-1}\right) ^{\omega (q)}\leq \frac{l}{k!}\leq \delta /4\text{.}
\label{r:one}
\end{eqnarray}

For $\omega ((q,d))=j<k$, $\omega (r)=\omega (q)-j>k$ (where $r=q/(q,d)$).
We now see that for $\delta >0$ small enough,%
\begin{equation}
|\tau (d^{\ast }d,q)|=1/\varphi (r)\leq 1/(p^{-}-1)^{k}=\log (1/\delta
)^{(-2-o(1))\log (1/\delta )}\leq \delta /4\text{,}  \label{r:oneB}
\end{equation}%
thus%
\begin{equation}
\sum_{j=0}^{k-1}\sum_{\omega ((d,q))=j}|w(d^{\ast }d)\tau (d^{\ast
}d,q)|\leq \frac{\delta }{4}\sum_{d^{\ast }d\in \mathcal{D}}|w(d^{\ast
}d)|=\delta /4.  \label{r:two}
\end{equation}

Relations (\ref{r:one})\ and (\ref{r:two})\ give $|A(q)|\leq \delta /2.$

(iii) Assume $2l<\omega (q)$. Then it is enough to see that for all $d^{\ast
}d\in \mathcal{D}$, $\omega (r)\geq l>k$. We now obtain in the same way as
in (\ref{r:oneB}) that $|\tau (d^{\ast }d,q)|\leq \delta /4$, but now for
all $d^{\ast }d\in \mathcal{D}$, thus $|A(q)|\leq \delta /4$.
\end{proof}

We now modify this for the exceptional case.

\begin{theorem}
\label{t:cancelling2}Assume $\delta >0$ is smaller than some $\delta _{0}>0$
and let $d_{D}$ be a positive integer, $d_{D}=\exp ((\log 1/\delta
)^{2+o(1)})$. Then there exists a collection of positive integers $\mathcal{D%
}$, \ such that $d_{D}|d$ for all $d\in \mathcal{D}$, not greater than $\exp
((\log 1/\delta )^{2+o(1)})$ and weights $w:\mathcal{D\rightarrow }%
\boldsymbol{R}$, $\sum_{d\in \mathcal{D}}w(d)=1$, such that for all positive
integers $q$,%
\begin{equation}
\sum_{d\in \mathcal{D}}w(d)\tau (d,q)\geq -\delta /2.  \label{r:B}
\end{equation}
\end{theorem}

\begin{proof}
We define $d^{\ast }=d_{D}\prod_{p\leq p^{-}}p$, where $p^{-}$ and all the
other constants remain the same as in the proof of Theorem \ref{t:cancelling}%
. Let $\mathcal{D}$ be the set of all the numbers $d^{\ast }d$, $d$
square-free, relatively prime with $d^{\ast }$, containing in its
decomposition only primes $p^{-}<p\leq p^{+}$, and such that $\omega (d)=l$.
The rest of the proof is analogous as the proof of Theorem \ref{t:cancelling}
with all calculations the same, thus omitted.
\end{proof}

\section{Proof of Theorem}

We complete the proof of Theorem \ref{t:main01} in this section. We will
choose below the constants $Q,R$, and will use the major arcs estimates for $%
q\leq Q$ and minor arcs estimates for $Q<q\leq R$. We will assume that $a/q$
is the Dirichlet's approximation of $\theta \in \boldsymbol{T}$, $|\theta
-a/q|\leq 1/qR$, $(a,q)=1$. The error terms in Propositions \ref{p:major}, %
\ref{p:main2} are then%
\begin{eqnarray*}
E_{1} &=&O\left( \frac{1}{D_{1}}+\frac{N}{R}\right) , \\
E_{2} &=&O\left( D_{1}N/R\right) ,
\end{eqnarray*}%
as $|\kappa |=1/qR$ and $d\leq D_{1}$. The error term for minor arcs is the
entire right-hand side of (\ref{p:m3}), thus as $q>Q$, it is%
\begin{equation*}
E_{3}=O\left( D_{1}^{2}(\log N)^{4}\left( \frac{1}{\sqrt{Q}}+N^{-1/5}+\frac{%
\sqrt{R}}{\sqrt{N}}\right) \right) .
\end{equation*}%
To complete the proof, we need to choose the constants $D_{1},N,Q,R\,$\ so
that the error terms $E_{1},E_{2},E_{3}\leq \delta /2$ for all $\theta \in 
\boldsymbol{T}$ on major; respectively minor arcs. As was noted in the
introduction, this is impossible, so we proceed as follows. We define%
\begin{equation*}
Q=\exp (\log (1/\delta )^{2+o(1)})
\end{equation*}%
(the constant obtained as the upper bound on $\mathcal{D}$ in\ Theorem \ref%
{t:cancelling}), and let%
\begin{equation*}
D_{0}=Q^{2},\text{ }D_{1}=Q^{4}\text{.}
\end{equation*}%
If $(D_{0},D_{1})$ is unexceptional, we construct the set $\mathcal{D}$
according to Theorem \ref{t:cancelling}, and if it is exceptional with the
modulus of the exceptional character $d_{D}\leq D_{0}$, then according to
Theorem \ref{t:cancelling2}. Now let $N_{0}=\exp (c_{2}(\log D_{1})^{2})$,
where $c_{2}$ is the constant in (\ref{r:new1}). We now define%
\begin{eqnarray*}
N_{j} &=&N_{0}D_{1}^{8j}, \\
R_{j}^{\ast } &=&N_{0}D_{1}^{8k+2}\text{,}
\end{eqnarray*}%
where $j=1,...,m$, $4/\delta \leq m<4/\delta +1$. Then for $0<\delta \leq
\delta _{0}$ for some $\delta _{0}$ small enough, and $j\leq j^{\ast }$, it
is easy to see that the error terms $E_{1},E_{2}\leq \delta /4$ for the
constants $Q,D_{1},N_{j},R_{j^{\ast }}^{\ast }$. Furthermore, if $j\geq
j^{\ast }+1$, the error term $E_{3}\leq \delta /4$ for the constants $%
Q,D_{1},N_{j},R_{j^{\ast }}^{\ast }$.

Let for a given $\theta \in \boldsymbol{T}$ the rational $a_{j}^{\ast
}/q_{j}^{\ast }$, $(a_{j}^{\ast },q_{j}^{\ast })$ be the Dirichlet's
approximation of $\theta $, $|\theta -a_{j}^{\ast }/q_{j}^{\ast }|\leq
1/q_{j}^{\ast }R_{j}^{\ast }$. Without loss of generality, we can also
assume that $a_{j}^{\ast }/q_{j}^{\ast }$ is the rational with the smallest $%
q_{j}^{\ast }$ for a given $R_{j}^{\ast }$. Then the sequence $q_{j}^{\ast }$
is increasing.

Let $j_{0}$ be the smallest index such that $q_{j_{0}}^{\ast }>Q$ ($%
q_{j_{0}}^{\ast }=m+1$ if $q_{j}^{\ast }\leq Q$ for all $j$). We define%
\begin{eqnarray*}
a_{j}/q_{j} &=&a_{j_{0}-1}^{\ast }/q_{j_{0}-1}^{\ast }\text{, }%
R_{j}=R_{j_{0}-1}^{\ast }\text{ for }j\leq j_{0}-1\text{,} \\
a_{j}/q_{j} &=&a_{j_{0}}^{\ast }/q_{j_{0}}^{\ast }\text{, }%
R_{j}=R_{j_{0}}^{\ast }\text{ for }j\geq j_{0}\text{.}
\end{eqnarray*}

Now one can easily check that for any $d\in \mathcal{D}$ and any $j\leq
j_{0}-1$, the assumptions of Proposition \ref{p:major} in the case $q$ not
dividing $d$, respectively of Proposition \ref{p:main2} in the case $q|d$,
do hold for the constants $D_{0},D_{1},Q,a_{j},q_{j},R_{j},N_{j}$, and as
was noted above, $E_{1},E_{2}\leq \delta /4$, thus%
\begin{equation}
F_{d,N_{j}}(\theta )\geq \tau (d,q_{j})-\delta /4\text{.}  \label{r:sum1}
\end{equation}%
Similarly for $j\geq j_{0}+1$ and $d\leq D_{1}$, the assumptions of
Corollary \ref{c:main3} hold and $E_{3}\leq \delta /4$, therefore%
\begin{equation}
F_{d,N_{j}}(\theta )\geq -\delta /4.  \label{r:sum2}
\end{equation}%
Also by definition,%
\begin{equation}
F_{d,N_{j_{0}}}(\theta )\geq -1\text{.}  \label{r:sum3}
\end{equation}%
Now the required polynomial is%
\begin{equation*}
T=\frac{1}{m}\sum_{d\in \mathcal{D}}\sum_{j=1}^{m}w(d)F_{d,N_{j}}.
\end{equation*}%
By applying (\ref{r:sum1}), (\ref{r:sum2}), (\ref{r:sum3}) for $1/m$ the sum
over $j$, and (\ref{r:A}) respectively (\ref{r:B}) for the sum over $d\in 
\mathcal{D}$, we get that for any $\theta \in \boldsymbol{T}$, $T(\theta
)\geq -\delta $. As the largest non-zero coefficient in $T$ is $dN_{m}\leq
N_{0}D_{1}^{8(4/\delta +1)+1}=\exp ((1/\delta )^{1+o(1)})$, this completes
the proof.

\section{The lower bound}

In this section we prove Theorem \ref{t:main02} on the lower bound for $%
\gamma (n)$ associated to the set $p-1$. Ruzsa in \cite{Ruzsa84b}, Section
5, constructed for a given $n$ a subset $A$ of integers not larger than $n$, 
$|A|\gg n^{((\log 2-\varepsilon )/\log \log n)}$ such that $A-A$ contains no
shifted prime $p-1$. We now construct a set $B$ of positive integers by the
following rule:\ if $x\equiv a(\func{mod}2n)$, then $x\in B$ for $a\in A$,
otherwise $x\not\in B$. Now clearly the upper Banach density of $B$ satisfies%
\begin{equation}
\rho (B)\gg n^{(-1+(\log 2-\varepsilon )/\log \log n)}  \label{r:upperbanach}
\end{equation}%
and $B$ contains no shifted prime $p-1$ smaller than $n$. Recall the measure
of intersectivity $\alpha (n)$ defined in the introduction, satisfying $%
\gamma \geq \alpha $. As $\alpha (n)$ is by definition $\gg $ than the
right-hand side of (\ref{r:upperbanach}), the proof is completed.

\section{Application:\ Heilbronn property of shifted primes}

An estimate for the Heilbronn property of shifted primes is an example of
application of Theorem \ref{t:main01}. If $\mathcal{H}$ is a set of positive
integers, we say that it is a Heilbronn set if $\eta =0$, where%
\begin{equation*}
\eta =\sup_{\theta \in \boldsymbol{T}}\inf_{h\in \mathcal{H}}||h\theta ||
\end{equation*}%
(for more detailed discussion, see \cite{Montgomery:94}, Section 2.7 or \cite%
{Schmidt:77}). One can quantify the Heilbronn property similarly as the van
der Corput and Poincar\'{e} properties of integers, and define%
\begin{equation}
\eta (n)=\sup_{\theta \in \boldsymbol{T}}\inf_{h\in \mathcal{H}%
_{n}}||h\theta ||,  \label{r:dnu}
\end{equation}%
where $\mathcal{H}_{n}=\mathcal{H}\cap \{1,...,n\}$. One can show that a set
is a Heilbronn set if and only if $\lim_{n\rightarrow \infty }\eta (n)=0$ (%
\cite{Montgomery:94}, Section 2.7). All van der Corput sets are Heilbronn
sets (the converse does not hold), and as was shown in \cite{Montgomery:94},
Theorem 2.9,%
\begin{equation}
\eta (n)\leq \gamma (n)\text{.}  \label{r:heilbronn}
\end{equation}

Various estimates for the function $\eta $ have been obtained by Schmidt 
\cite{Schmidt:77} for sets of values of polynomials with integer
coefficients. An upper bound for the set of shifted primes follows from\
Theorem \ref{t:main01} and (\ref{r:heilbronn}).

\begin{corollary}
If $\eta $ is the arithmetic function (\ref{r:dnu}) associated to the set of
shifted primes $\mathcal{H}$, then $\eta (n)=O((\log n)^{-1+o(1)})$.
\end{corollary}

\begin{acknowledgement}
The author thanks the anonymous referee for suggesting an improvement of an
early version of the paper which substantially improved the upper bound in
the main result.
\end{acknowledgement}

\end{document}